\documentclass[12pt,reqno]{amsart}
\usepackage[english]{babel}
\usepackage{amsmath,amsfonts,amssymb,amsthm,amscd,latexsym}
\usepackage[T1]{fontenc}
\usepackage[utf8]{inputenc}

\usepackage{tikz}
\usetikzlibrary{arrows}
\usepackage{color}
\usepackage{cite}
\usepackage{multirow}
\usepackage{cite}
\usepackage{xcolor}
\newtheorem{theorem}{Theorem}[section]

\newtheorem{lemma}[theorem]{Lemma}

\newtheorem{definition}[theorem]{Definition}

\numberwithin{equation}{section}
\DeclareFontFamily{OT1}{pzc}{}
\DeclareFontShape{OT1}{pzc}{m}{it}{<-> s * [1.10] pzcmi7t}{}
\DeclareMathAlphabet{\mathpzc}{OT1}{pzc}{m}{it}

\begin{document}

 \bigskip
	
	\begin{center}
		\textbf{Local and 2-local $\frac{1}2$-derivations of  infinite-dimensional Lie algebras } \\
		\medskip
		\textbf{Shavkat Ayupov$^{1,2}$, Abdireymov Arislanbay$^{1}$, Bakhtiyor Yusupov$^{1,3}$} \\
		\smallskip
		{$^{1}$V.I.Romanovskiy Institute of Mathematics Uzbekistan Academy of Sciences, Tashkent, Uzbekistan}\\
        		  {$^{2}$ National University of Uzbekistan, Tashkent, Uzbekistan}\\
                   {$^{3}$ Department of Algebra and Mathematical engineering, Urgench State University, Urgench, Uzbekistan}\\
    {shavkat.ayupov@mathinst.uz, \ aaabdireymov1001@gmail.com,\  baxtiyor\_yusupov\_93@mail.ru}\\
	\end{center}

	\bigskip

\noindent {\bf Abstract.}
{\it In this work, we describe local and 2-local $\frac12$-derivations of infinite-dimensional Lie algebras. We prove that all local and 2-local $\frac12$-derivations of the Witt
algebra as well as of the positive Witt algebra and the classical one-sided Witt algebra are $\frac12$-derivations. We also give an example of an infinite-dimensional Lie algebra with a local (2-local) $\frac12$-derivation which is not a $\frac12$-derivation. Further we prove
that all local (2-local) $\frac12$-derivations on the $\mathcal{W}(a,b)$ algebra  are $\frac12$-derivations.
}

\bigskip



 \bigskip

\bigskip

\section{Introduction}

In recent years non-associative analogues of classical constructions become of
interest in connection with their applications in many branches of mathematics and
physics. The notions of local and $2$-local derivations are also become popular for some non-associative algebras such as the Lie and Leibniz algebras.

The notions of local derivations were introduced in 1990 by R.V.Kadison \cite{Kadison} and D.R.Larson, A.R.Sourour \cite{Larson}. Later in 1997, P. $\check{S}$emrl introduced the notions of  $2$-local derivations and
$2$-local automorphisms on algebras \cite{Sem}. The main problems concerning these notions are to find conditions under which all local ($2$-local) derivations become (global)derivations and to present examples of algebras with local ($2$-local) derivations that
are not derivations.

Investigation of local and 2-local derivations on Lie algebras was initiated in papers \cite{Ayupov1} and \cite{Ayupov2}. Sh.A. Ayupov and K.K. Kudaybergenov have proved that every local(2-local) derivation on a semi-simple Lie algebra is a derivation and gave examples of nilpotent finite-dimensional Lie algebras with 2-local derivations which are not derivations. Let us present a list of infinite dimensional  algebras for which all 2-local derivations are derivations:  infinite dimensional Witt  algebras over an algebraically closed field of characteristic zero \cite{AKY};  locally finite split simple Lie algebras over a field of characteristic zero \cite{AKYu1}; Virasoro algebras \cite{AY2}; $W(2,2)$ Lie algebras \cite{Tang}; Virasoro-like algebra \cite{XMP}; the Schrodinger-Virasoro algebra \cite{QX}; Jacobson-Witt algebras \cite{YK};  deformative super $W$-algebras $W_{\lambda}^s(2,2)$ \cite{Abd}; planar Galilean conformal algebra  \cite{ChenXe}.

The investigation of local and $2$-local $\delta$-derivations on Lie algebras was initiated by A.~Khudoyberdiyev and B.~Yusupov in \cite{KhudYus}. In that paper, the notions of local and $2$-local $\delta$-derivations were introduced and local and $2$-local $\tfrac12$-derivations were described for finite-dimensional solvable Lie algebras with filiform, Heisenberg, and abelian nilradicals. Moreover, local $\tfrac12$-derivations were described for oscillator Lie algebras, conformal perfect Lie algebras, and Schr\"{o}dinger algebras. A related problem for Leibniz algebras was considered in \cite{Yusupov21} by U.~Mamadaliyev, A.~Sattarov, and B.~Yusupov. They proved that any local $\tfrac12$-derivation on solvable Leibniz algebras with model or abelian nilradicals and with maximal complementary space dimension is a $\tfrac12$-derivation. They also obtained results for solvable Leibniz algebras with abelian nilradicals and one-dimensional complementary space, and studied the corresponding $2$-local case.  Similar results for local $\tfrac12$-derivations of naturally graded quasi-filiform Leibniz algebras of type~I were obtained by B.~Yusupov, N.~Vaisova, and T.~Madrakhimov in \cite{YusVai}; in particular, they showed that such algebras as a rule admit local $\tfrac12$-derivations which are not $\tfrac12$-derivations. In \cite{Yus1,Yus2}, local and $2$-local $\tfrac12$-derivations on $p$-filiform Leibniz algebras were studied, and it was shown that, in general, $p$-filiform Leibniz algebras admit local (and $2$-local) $\tfrac12$-derivations which are not $\tfrac12$-derivations.

In this work, we investigate local and 2-local $\delta$-derivations of Lie algebras. After Preliminaries (Section 2), in the third section, we describe $\frac12$-derivation of the positive and one-sided Witt algebras. In the fourth section, we prove that all local $\frac12$-derivations of the Witt algebra, as well as of the positive Witt algebra and the classical one-sided Witt algebra and also of the Lie algebra $\mathcal{W}(a,b)$
are $\frac12$ derivations. We also give an example of an infinite-dimensional Lie algebra with a local $\frac12$-derivation which is not a $\frac12$-derivation.
In the fifth section, we prove similar results for 2-local $\frac12$-derivations on the same  algebras listed in previous section. We also provide an example of an infinite-dimensional Lie algebra with a 2-local $\frac{1}{2}$-derivation that is not a $\frac{1}{2}$-derivation.
In the sixth section, we prove that all 2-local $\frac12$-derivations of the solvable infinite-dimensional Lie algebra with an abelian radical of codimension $1$ are $\frac12$-derivations.

\section{Preliminaries}

\medskip

In this section, we give some necessary definitions and preliminary
results.

 Omitted products in the multiplication table of an algebra are assumed to be zero. Moreover, due to the anti-commutativity of Lie algebras, symmetric products for these algebras are also omitted.

\begin{definition}\label{12der}
		Let $({\mathfrak L}, [-,-])$ be a Lie algebra over a field $\mathbb{F}$ with a multiplication $[-,-].$ A linear map $\varphi$ is called a $\delta$-derivation if it satisfies
\begin{center}
$\varphi[x,y]= \delta \big([\varphi(x),y]+ [x, \varphi(y)] \big),$
\end{center}
where $\delta$ is fixed and it is from the ground field $\mathbb{F}.$
\end{definition}

Note that $1$-derivation is a usual derivation and $(-1)$-derivation is called an anti-derivation.
If $\varphi_1$ and $\varphi_2$ are $\delta_1$ and $\delta_2$-derivations, respectively, then their commutator $[\varphi_1, \varphi_2] = \varphi_1\varphi_2 - \varphi_2\varphi_1$ is a $\delta_1\delta_2$-derivation.  The set of all $\delta$-derivations, for the fixed $\delta$, we denote by $\operatorname{Der}_{\delta}(\mathfrak L).$

Since the notions of local operators can be defined for any type of operator, the notions of local and $2$-local $\delta$-derivations are defined as follows:

\begin{definition}
A linear map $\Delta$ is called a local $\delta$-derivation, if for any $x \in \mathfrak L,$ there exists a $\delta$-derivation $\varphi_x: \mathfrak L \rightarrow \mathfrak L$ (depending on $x$) such that $\Delta(x) = \varphi_x(x).$ The set of all local $\delta$-derivations on $\mathfrak L$ we denote by $\mathrm{Loc}\mathrm{Der}_{\delta}(\mathfrak L).$
\end{definition}

\begin{definition}
A map $\nabla :  \mathfrak L \rightarrow \mathfrak L$ (not
necessary  linear) is called a $2$-local $\delta$-derivation, if for any $x,y\in \mathfrak L$, there exists a $\delta$-derivation $\varphi_{x,y}\in \mathrm{Der}_{\delta}
(\mathfrak L)$ such that
{\small\[
\nabla(x)=\varphi_{x,y}(x), \quad \nabla(y)=\varphi_{x,y}(y).
\]}
\end{definition}

It should be noted that $2$-local $\delta$-derivation is not necessarily linear, but for any $x \in \mathfrak L$ and for any scalar $\lambda,$ we have
$$\nabla(\lambda x) = \varphi_{x, \lambda x}(\lambda x) = \lambda  \varphi_{x, \lambda x}(x) =
\lambda \nabla(x).$$

In this work, we focus on investigating local and $2$-local $\frac 1 2$-derivations. Since any $\frac 1 2$-derivation is a local and $2$-local $\frac 1 2$-derivations, we are interesting
on local and $2$-local $\frac 1 2$-derivation, which is not a $\frac 1 2$-derivation. Such local (resp. $2$-local) $\frac 1 2$-derivations we call non-trivial local (resp. $2$-local) $\frac 1 2$-derivations.

We now give definitions of some algebras studied in this paper.

Let $A=\mathbb{C}[x,x^{-1}]$ be the algebra of all Laurent polynomials in one variable over a field of
characteristic zero $\mathbb{F}.$ The Lie algebra of derivations
$$
Der(A)=span\left\{f(x)\frac{d}{dx}:f\in\mathbb{C}\left[x,x^{-1}\right]\right\}
$$
with the Lie bracket is called a \emph{Witt algebra} and denoted by \(\mathcal{W}\).
Then \cite{Kac} $\mathcal{W}$ is   an  infinite-dimensional simple algebra  which  has the basis $\left\{e_i: e_i=x^{i+1}\frac{d}{dx}, i\in \mathbb{Z}\right\}$ and the
multiplication rule
$$
[e_i, e_j]=(j-i)e_{i+j},\, i,j \in \mathbb{Z}.
$$
We also consider  the infinite-dimensional  positive part \(\mathcal{W}^+\) of the Witt algebra. The \emph{positive Witt algebra} \(\mathcal{W}^+\) is an infinite-dimensional
Lie algebra \cite{Dim} which has the basis
 $\left\{e_i: e_i=x^{i+1}\frac{d}{dx}, i\in \mathbb{N}\right\}$ and the
multiplication rule
$$
[e_i, e_j]=(j-i)e_{i+j},\, i,j \in \mathbb{N}.
$$

The classical \emph{one-sided Witt algebra} $\mathcal{W}_1$ is defined with basis $\{e_i|\ i\in\mathbb{Z},\ i\geq -1\}$ and brackets
$$
[e_i, e_j]=(j-i)e_{i+j},\, \ \text{for all}\ \ i,j\geq -1,\  i,j\in \mathbb{Z}.
$$
Thus by realizing $e_i$ as $x^{i+1}\frac{d}{dx},$ one immediately observes that $\mathcal{W}_1=Der\mathbb{C}[x]=\mathbb{C}[x]\frac{d}{dx}$ is the derivation algebra of the polynomial algebra $\mathbb{C}[x].$

In \cite{Kac}, a class of representations $I(a,b)=\oplus_{i\in\mathbb{Z}}\mathbb{C}f_i$ for the Witt algebra $\mathcal{W}$ with two complex parameters $a$ and $b$ had been introduced. The action of $\mathcal{W}$ on $I(a,b)$ is given by
$e_i\cdot f_j= -(i+a+bj)f_{i+j}.$ $I(a,b)$ is the so called tensor density module. The Lie algebra $\mathcal{W}(a,b) = \mathcal{W}\ltimes I(a,b),$ where $\mathcal{W}$ is the Witt algebra and $I(a,b)$ is the tensor density module of $W.$

The Lie algebra $\mathcal{W}(a,b)$ is spanned by generators $\{ e_i, f_j\}_{i,j\in\mathbb{Z}}.$ These generators satisfy
$$
[e_i,e_j]=(i-j)e_{i+j},\quad [e_i,f_j]=-(i+a+bj)f_{i+j}.
$$

\section{$\frac12$-Derivations of the Witt algebras}

In this section, we study $\frac12$-derivations of the Witt, positive and one-sided Witt algebras.

In \cite{BKL}, the general structure of $\frac12$-derivations on the Witt algebra is determined.

\begin{theorem}\label{BKL}
Let $\varphi$ be $\frac{1}{2}$-derivation of the Witt algebra $\mathcal{W}$. Then there is a set $\{\alpha_i\}_{i\in\mathbb{Z}}$ of
elements from the basic field, such that $\varphi(e_i)=\sum\limits_{j\in\mathbb{Z}}\alpha_je_{i+j}.$
\end{theorem}

Now we give a classification of derivations of $\mathcal{W}_1$  one-sided Witt algebras.
\begin{theorem}\label{onewitt}
    Any $\frac{1}{2}$-derivation on $\mathcal{W}_1$ has the following form
\begin{equation*}
\varphi(e_i)=\sum\limits_{j\ge 0}\alpha_{j} e_{j+i}, \quad i\geq -1.
\end{equation*}
Where set $\{\alpha_i\}_{i\in\mathbb{N}}$ of
elements from the basic field.
\end{theorem}

\begin{proof}
The proof consists of the following three main lemmas.

For $k \geq 0$, define $T_k: \mathcal{W}_1 \to \mathcal{W}_1$ by $T_k(e_i) = e_{i+k}$ for all $i \geq -1$.

\begin{lemma}\label{lemtk}
Each $T_k$ is a $\frac{1}{2}$-derivation of $\mathcal{W}_1$.
\end{lemma}

\begin{proof}
For basis elements $e_i, e_j$ with $i, j \geq -1$,
\begin{align*}
[T_k(e_i), e_j] + [e_i, T_k(e_j)]
&= [e_{i+k}, e_j] + [e_i, e_{j+k}] \\
&= (j - (i+k))e_{i+j+k} + ((j+k) - i)e_{i+j+k} \\
&= 2(j-i)e_{(i+j)+k} \\
&= 2T_k([e_i, e_j]).
\end{align*}
By linearity, this holds for all $x, y \in \mathcal{W}_1$.
\end{proof}

Since linear combinations of $\frac{1}{2}$-derivations are again $\frac{1}{2}$-derivations, any operator of the form
\begin{equation}\label{eq:D}
D = \sum_{k \geq 0} a_k T_k, \quad a_k \in \mathbb{C},
\end{equation}
is a $\frac{1}{2}$-derivation.

Let $\varphi$ be an arbitrary $\frac{1}{2}$-derivation. Write
\begin{equation*}
\varphi(e_i) = \sum_{j \geq -1} \alpha_{i,j} e_j, \quad i \geq -1.
\end{equation*}

Applying the $\frac{1}{2}$-derivation condition to the pair $(e_0, e_i)$, we obtain
\begin{equation*}
[\varphi(e_0), e_i] + [e_0, \varphi(e_i)] = 2\varphi([e_0, e_i]) = 2i\varphi(e_i).
\end{equation*}

Computing each term, we have
\begin{align*}
[\varphi(e_0), e_i] &= \sum_{j \geq -1} \alpha_{0,j} [e_j, e_i] = \sum_{j \geq -1} \alpha_{0,j}(i-j)e_{i+j}, \\
[e_0, \varphi(e_i)] &= \sum_{j \geq -1} \alpha_{i,j} [e_0, e_j] = \sum_{j \geq -1} \alpha_{i,j} \cdot j \cdot e_j.
\end{align*}

Therefore,
\begin{equation}\label{eq:main}
\sum_{j \geq -1} \alpha_{0,j}(i-j)e_{i+j} + \sum_{j \geq -1} j\alpha_{i,j}e_j = 2i\sum_{j \geq -1} \alpha_{i,j}e_j.
\end{equation}

Rearranging the second sum
\begin{equation}\label{eq:rearranged}
\sum_{j \geq -1} \alpha_{0,j}(i-j)e_{i+j} = \sum_{j \geq -1} (2i-j)\alpha_{i,j}e_j.
\end{equation}

\begin{lemma}\label{lem:coeff}
We have $\alpha_{0,-1} = 0$ and for all $i \geq -1$,
\begin{equation*}
\varphi(e_i) = \sum_{j \geq 0} \alpha_{0,j} e_{i+j} + \delta_i \beta_i e_{2i},
\end{equation*}
where $\delta_i = 1$ for $i \geq 1$ and $\delta_i = 0$ for $i = -1, 0$, and $\beta_i = \alpha_{i,2i} - \alpha_{0,2i}$.
\end{lemma}

\begin{proof}
Setting $i = -1$ in \eqref{eq:rearranged}, we have
\begin{equation*}
\sum_{j \geq -1} \alpha_{0,j}(-1-j)e_{j-1} = \sum_{j \geq -1} (-2-j)\alpha_{-1,j}e_j.
\end{equation*}

The left side contains $e_{-2}$ (when $j = -1$) with coefficient $0$, but the right side has no $e_{-2}$ term. This forces $\alpha_{0,-1} = 0$.

For $j \neq 2i$, comparing coefficients of $e_j$ in \eqref{eq:rearranged} gives $\alpha_{i,j} = \alpha_{0,j-i}$. For $j = 2i$, we obtain the stated form with $\beta_i = \alpha_{i,2i} - \alpha_{0,2i}$.
\end{proof}

Now define $D$ as in \eqref{eq:D} with $a_k = \alpha_{0,k}$. Then $D(e_0) = \varphi(e_0)$ and $D(e_{-1}) = \varphi(e_{-1})$. Set $\varphi_1 = \varphi - D$. Then $\varphi_1$ is a $\frac{1}{2}$-derivation satisfying $\varphi_1(e_{-1}) = \varphi_1(e_0) = 0$ and
\begin{equation*}
\varphi_1(e_i) = \beta_i e_{2i}, \quad i \geq 1.
\end{equation*}

\begin{lemma}\label{lemphi}
$\varphi_1 = 0$, i.e., $\beta_i = 0$ for all $i \geq 1$.
\end{lemma}

\begin{proof}
Apply the $\frac{1}{2}$-derivation condition to the pair $(e_{-1}, e_{i+1})$ for $i \geq 1$,
\begin{equation*}
2\varphi_1([e_{-1}, e_{i+1}]) = [\varphi_1(e_{-1}), e_{i+1}] + [e_{-1}, \varphi_1(e_{i+1})].
\end{equation*}

Since $[e_{-1}, e_{i+1}] = (i+2)e_i$, $\varphi_1(e_{-1}) = 0$, and $\varphi_1(e_{i+1}) = \beta_{i+1}e_{2i+2}$,
\begin{equation*}
2(i+2)\beta_i e_{2i} = [e_{-1}, \beta_{i+1}e_{2i+2}] = \beta_{i+1}(2i+3)e_{2i+1}.
\end{equation*}

Since $\{e_{2i}, e_{2i+1}\}$ are linearly independent, both sides must vanish
\begin{equation*}
(i+2)\beta_i = 0 \quad \text{and} \quad (2i+3)\beta_{i+1} = 0.
\end{equation*}

For $i \geq 1$, we have $i+2 \neq 0$, hence $\beta_i = 0$.
\end{proof}

By Lemma \ref{lemphi}, $\varphi = D = \sum_{j \geq 0} \alpha_j T_j$ as required.
\end{proof}

We now state the following theorem for the positive Witt algebra.

\begin{theorem}\label{prop:half-der-Wplus-generators}
Any $\frac12$-derivation $\varphi$ of positive Witt algebras $\mathcal{W}^{+}$ has the form
\begin{equation}\label{eq:halfWplus-form-gen}
\varphi(e_i)=\sum_{k\ge i}\alpha_{k-i+1}\,e_{k}\qquad (i\in\mathbb{N}),
\end{equation}
where set $\{\alpha_i\}_{i\in\mathbb{N}}$ of
elements from the basic field.

\end{theorem}
\begin{proof}
    The proof is similar to Theorem \ref{onewitt}.
\end{proof}

The description of $\frac12$-derivations of the Lie algebra $\mathcal{W}(a,b)$ was established in \cite{BKL}.
\begin{theorem}\label{thm52}
There are no non-trivial $\frac{1}{2}$-derivations of the Lie algebra $\mathcal{W}(a,b)$ for $b\neq -1.$
Let $\varphi$ be $\frac{1}{2}$-derivation of the algebra $\mathcal{W}(a,-1)$, then there are two finite sets of elements
from the basic field $\{\alpha_t\}_{t\in\mathbb{Z}}$ and $\{\beta_t\}_{t\in\mathbb{Z}}$, such that
$$\varphi(e_i)=\sum\limits_{t\in\mathbb{Z}}\alpha_te_{i+t}+\sum\limits_{t\in\mathbb{Z}}\beta_tf_{i+t},
$$
and
$$
\varphi(f_i)=\sum\limits_{t\in\mathbb{Z}}\alpha_tf_{i+t}.
$$
\end{theorem}

\section{Local $\frac12$-derivations of infinite-dimensional Lie algebras}

In this section, we study local  $\frac12$-derivations of Witt, positive and one-sided Witt algebra and also thin Lie algebras.

\subsection{Local $\frac12$-derivations of Witt algebras.}

Now we shall give the main result concerning local $\frac12$-derivations of Witt algebras.

\begin{theorem}\label{locwitt}
Any local $\frac{1}{2}$-derivation of the Witt algebra $\mathcal{W}$ is a $\frac{1}{2}$-derivation.
\end{theorem}

For the proof of this theorem, we need several lemmas.
\begin{lemma}\label{Lem1}
Let $\Delta$ be a local $\frac{1}{2}$-derivation on $\mathcal{W}$. If $\Delta(e_m)=0$, then $\Delta(e_{m+1})=0$ where $m\in\mathbb{Z}$.
\end{lemma}

\begin{proof}
By Theorem \ref{BKL}, any $\frac{1}{2}$-derivation $\varphi$ on $\mathcal{W}$ satisfies
\[
\varphi(e_m)=\sum_{j\in\mathbb{Z}}\alpha_j e_{m+j}.
\]
Hence, there exist integers $s\le t$ such that
\[
\varphi(e_m)=\sum_{j=s}^{t}\alpha_j e_{m+j}.
\]

Consider the equality
\[
\varphi_1(e_{m+1})-x\varphi_1(e_m)
=\Delta(e_{m+1}-x e_m)
=\Delta(e_{m+1})
=\varphi_2(e_{m+1}),
\]
where $x\in\mathbb{C}^*$ and $\varphi_1,\varphi_2$ are $\frac{1}{2}$-derivations associated with $\Delta$.

Then
\[
\sum_{j=s}^{t}\alpha_j e_{m+1+j}
-
x\sum_{j=s}^{t}\alpha_j e_{m+j}
=
\sum_{j=s'}^{t'}\beta_j e_{m+1+j},
\]
where $s'=s-1$ and $t'=t$.

Writing this equality with respect to the basis $\{e_k\}$, we obtain
\[
(-x\alpha_s-\beta_{s-1})e_{m+s}
+\sum_{j=s}^{t-1}(\alpha_j-x\alpha_{j+1}-\beta_j)e_{m+1+j}
+(\alpha_t-\beta_t)e_{m+1+t}=0.
\]

Consequently,
\[
\beta_t+x^{-1}\beta_{t-1}+x^{-2}\beta_{t-2}
+\cdots
+x^{-(t-s)}\beta_s
+x^{-(t-s+1)}\beta_{s-1}=0.
\]

Since this equality holds for all $x\in\mathbb{C}^*$ and the coefficients
$\beta_j$ $(s-1\le j\le t)$ are independent of $x$, it follows that
\[
\beta_j=0,\qquad s-1\le j\le t.
\]

Therefore, $\Delta(e_{m+1})=0$.
\end{proof}

\begin{lemma}\label{Lem2}
Let $\Delta$ be a local $\frac{1}{2}$-derivation on $\mathcal{W}$. If $\Delta(e_m)=0$, then $\Delta(e_{m-1})=0$ where $m\in\mathbb{Z}$.
\end{lemma}

\begin{proof}
The proof is similar to that of Lemma~\ref{Lem1}.
\end{proof}

\begin{lemma}\label{witte_0}
Let $\Delta$ be a local $\frac{1}{2}$-derivation on $\mathcal{W}$. If $\Delta(e_0)=0$, then
\[
\Delta(e_m)=0 \quad \text{for all } m\in\mathbb{Z}.
\]
\end{lemma}

\begin{proof}
Assume $\Delta(e_0)=0$. By Lemma~\ref{Lem1}, using the induction, we obtain
\[
\Delta(e_m)=0 \quad \text{for all } m>0.
\]

Similarly, applying Lemma~\ref{Lem2} yields
\[
\Delta(e_m)=0 \quad \text{for all } m<0.
\]

Hence, $\Delta(e_m)=0$ for all $m\in\mathbb{Z}$.
\end{proof}

\textit{Proof of Theorem \ref{locwitt}.} Let $\Delta$ be any local $\frac12$-derivation on $\mathcal{W}$. Then there exists a $\frac12$-derivation $\varphi$ on $\mathcal{W}$ such that
$$\Delta(e_0)=\varphi(e_0)$$

Set $\Delta_1=\Delta-\varphi$. Then $\Delta_1$ is a local $\frac12$-derivation such that $\Delta_1(e_0)=0$.
By Lemma \ref{witte_0}, $\Delta_1(e_m)=0$ for all $m\in\mathbb{Z}$. Then $\Delta_1\equiv0.$
Thus $\Delta=\varphi$ is a $\frac12$-derivation. The proof is complete. \hfill $\square$

\subsection{Local $\frac12$-derivation on subalgebras of the Witt algebra.}
In this subsection, we study local $\frac12$-derivations of positive Witt and
one-sided Witt algebras.

Now we shall give the main result concerning local  $\frac12$-derivations of positive Witt and one-sided Witt algebras.

\begin{theorem}\label{locwittpos}
Any local $\frac12$-derivation of $\mathcal{W}^+$ and $\mathcal{W}_1$ is a $\frac12$-derivation.
\end{theorem}
 We prove the theorem for the algebra $\mathcal{W}^+$, and for the algebra $\mathcal{W}_1$, the proof is similar. For the proof of this theorem, we need several lemmas.

\begin{lemma}\label{PLem1}
Let $\Delta$ be a local $\frac{1}{2}$-derivation on $\mathcal{W}^+$. If $\Delta(e_m)=0$, then $\Delta(e_{m+1})=0$ where $m\in\mathbb{N}$.
\end{lemma}

\begin{proof}
By Theorem \ref{prop:half-der-Wplus-generators}, any $\frac{1}{2}$-derivation $\varphi$ on $\mathcal{W}^+$ satisfies
\[
\varphi(e_m)=\sum_{j\geq m}\alpha_{j-m+1} e_{j}.
\]
Hence, there exist integers $1\le t$ such that
\[
\varphi(e_m)=\sum_{j=1}^{t}\alpha_j e_{m+j-1}.
\]

Consider the equality
\[
\varphi_1(e_{m+1})-x\varphi_1(e_m)
=\Delta(e_{m+1}-x e_m)
=\Delta(e_{m+1})
=\varphi_2(e_{m+1}),
\]
where $x\in\mathbb{C}^*$ and $\varphi_1,\varphi_2$ are $\frac{1}{2}$-derivations associated with $\Delta$.

Then
\[
\sum_{j=1}^{t}\alpha_j e_{m+j}
-
x\sum_{j=1}^{t}\alpha_j e_{m+j-1}
=
\sum_{j=1}^{t'}\beta_j e_{m+j},
\]
where $t'=t$.

Writing this equality with respect to the basis $\{e_k\}$, we obtain
\[
-x\alpha_1e_{m}
+\sum_{j=1}^{t-1}(\alpha_j-x\alpha_{j+1}-\beta_j)e_{m+j}
+(\alpha_t-\beta_t)e_{m+t}=0.
\]

Consequently,
\[
\beta_t+x^{-1}\beta_{t-1}+x^{-2}\beta_{t-2}
+\cdots
+x^{-(t-2)}\beta_2
+x^{-(t-1)}\beta_{1}=0.
\]

Since this equality holds for all $x\in\mathbb{C}^*$ and the coefficients
$\beta_j$ $(s-1\le j\le t)$ are independent of $x$, it follows that
\[
\beta_j=0,\qquad 1\le j\le t.
\]

Therefore, $\Delta(e_{m+1})=0$.
\end{proof}

\begin{lemma}\label{pwitte_1}
Let $\Delta$ be a local $\frac{1}{2}$-derivation on $\mathcal{W}^+$. If $\Delta(e_1)=0$, then
\[
\Delta(e_m)=0 \quad \text{for all } m\in\mathbb{N}.
\]
\end{lemma}

\begin{proof}
Assume $\Delta(e_1)=0$. By Lemma~\ref{PLem1}, using the induction, we obtain
\[
\Delta(e_m)=0 \quad \text{for all } m>1.
\]

Hence, $\Delta(e_m)=0$ for all $m\in\mathbb{N}$.
\end{proof}

\textit{Proof of Theorem \ref{locwittpos}. }Let $\Delta$ be any local $\frac12$-derivation on $\mathcal{W}^+$. Then there exists a $\frac12$-derivation $\varphi$ on $\mathcal{W}^+$ such that
$$\Delta(e_1)=\varphi(e_1)$$

Set $\Delta_1=\Delta-\varphi$. Then $\Delta_1$ is a local $\frac12$-derivation such that $\Delta_1(e_1)=0$.
By Lemma \ref{pwitte_1}, $\Delta_1(e_m)=0$ for all $m\in\mathbb{N}$. Then $\Delta_1\equiv0.$
Thus $\Delta=\varphi$ is a $\frac12$-derivation. The proof is complete. \hfill $\square$

\subsection{Local $\frac12$-derivation on $\mathcal{W}(a,b)$.}
In this subsection, we study local $\frac{1}{2}$-derivation on Lie algebra $\mathcal{W}(a,b)$.

Now we shall give the main result concerning local $\frac{1}{2}$-derivation of Lie algebra $\mathcal{W}(a,b)$.

\begin{theorem}\label{locwitt(a,b)}
Any local $\frac{1}{2}$-derivation of Lie algebra $\mathcal{W}(a,b)$ is a $\frac{1}{2}$-derivation.
\end{theorem}
For the proof of this theorem, we need several lemmas.

\begin{lemma}\label{Lem100}
Let $\Delta$ be a local $\frac{1}{2}$-derivation on $\mathcal{W}(a,-1)$. If $\Delta(e_m)=0$, then $\Delta(e_{m+1})=0$ where $m\in\mathbb{Z}$.
\end{lemma}

\begin{proof}
By Theorem \ref{thm52}, any $\frac{1}{2}$-derivation $\varphi$ on $\mathcal{W}(a,-1)$ satisfies
\[
\varphi(e_m)=\sum_{j\in\mathbb{Z}}\alpha_j e_{m+j}+\sum_{j\in\mathbb{Z}}\beta_j f_{m+j}.
\]
Hence, there exist integers $s\le t, \, p\le q$ such that
\[
\varphi(e_m)=\sum_{j=s}^{t}\alpha_j e_{m+j}+\sum_{j=p}^{q}\beta_j f_{m+j}.
\]

Consider the equality
\[
\varphi_1(e_{m+1})-x\varphi_1(e_m)
=\Delta(e_{m+1}-x e_m)
=\Delta(e_{m+1})
=\varphi_2(e_{m+1}),
\]
where $x\in\mathbb{C}^*$ and $\varphi_1,\varphi_2$ are $\frac{1}{2}$-derivations associated with $\Delta$.

Then
\[
\sum_{j=s}^{t}\alpha_j e_{m+1+j}
-
x\sum_{j=s}^{t}\alpha_j e_{m+j}+\sum_{j=p}^{q}\beta_j f_{m+1+j}
-
x\sum_{j=p}^{q}\beta_j f_{m+j}
=
\sum_{j=s'}^{t'}\alpha'_j e_{m+1+j}+\sum_{j=p'}^{q'}\beta'_j f_{m+1+j},
\]
where $s'=s-1$ and $t'=t$, $p'=p-1$ and $q'=q$.

Writing this equality with respect to the basis $\{e_k\}$, we obtain
\[
(-x\alpha_s-\alpha'_{s-1})e_{m+s}
+\sum_{j=s}^{t-1}(\alpha_j-x\alpha_{j+1}-\alpha'_j)e_{m+1+j}
+(\alpha_t-\alpha'_t)e_{m+1+t}=0.
\]

Consequently,
\[
\alpha'_t+x^{-1}\alpha'_{t-1}+x^{-2}\alpha'_{t-2}
+\cdots
+x^{-(t-s)}\alpha'_s
+x^{-(t-s+1)}\alpha'_{s-1}=0.
\]

Since this equality holds for all $x\in\mathbb{C}^*$ and the coefficients
$\alpha'_j$ $(s-1\le j\le t)$ are independent of $x$, it follows that
\[
\alpha'_j=0,\quad s-1\le j\le t.
\]
Similarly, we find that $\beta'_j=0$, where $p-1\le j\le q$.
Therefore, $\Delta(e_{m+1})=0$.
\end{proof}

\begin{lemma}\label{Lem200}
Let $\Delta$ be a local $\frac{1}{2}$-derivation on $\mathcal{W}(a,-1)$. If $\Delta(e_m)=0$, then $\Delta(e_{m-1})=0$ where $m\in\mathbb{Z}$.
\end{lemma}

\begin{proof}
The proof is similar to that of Lemma~\ref{Lem100}.
\end{proof}

\begin{lemma}\label{300}
Let $\Delta$ be a local $\frac{1}{2}$-derivation on $\mathcal{W}(a,-1)$. If $\Delta(e_0)=0$, then
\[
\Delta(e_m)=0 \quad \text{for all } m\in\mathbb{Z}.
\]
\end{lemma}

\begin{proof}
Assume $\Delta(e_0)=0$. By Lemma~\ref{Lem100}, using mathematical induction, we obtain
\[
\Delta(e_m)=0 \quad \text{for all } m>0.
\]

Similarly, applying Lemma~\ref{Lem200} yields
\[
\Delta(e_m)=0 \quad \text{for all } m<0.
\]

Hence, $\Delta(e_m)=0$ for all $m\in\mathbb{Z}$.
\end{proof}

\begin{lemma}\label{Lem400}
    Let $\Delta$ be a local $\frac{1}{2}$-derivation on $\mathcal{W}(a,-1)$. If $\Delta(e_m)=0$ for any $m \in \mathbb{Z}$, then
\[
\Delta(f_m)=0 \quad \text{for all } m\in\mathbb{Z}.
\]
\end{lemma}
\begin{proof}
The $\frac{1}{2}$-derivation of the algebra $\mathcal{W}(a,-1)$ is written as follows by finding integers $s \le t$ and $p \le q$ according to Theorem \ref{thm52}
\[
\varphi(e_m)=\sum_{j=s}^{t}\alpha_j e_{m+j}+\sum_{j=p}^{q}\beta_j f_{m+j},
\]
\[
\varphi(f_m)=\sum_{j=s}^{t}\alpha_j f_{m+j}.
\]
Now, we calculate the value of the local $\frac12$-derivation at $f_m$. There exist integers $p' \le q'$ such that the $\frac12$-derivation $\varphi_1$ is given by
\[
\Delta(f_m)=\varphi_1(f_m)=\sum_{j=p'}^{q'}\beta'_j f_{m+j}.
\]

Consider the equality
\[
\varphi(f_m)+\varphi(e_m)+\varphi(e_k)
=\Delta(f_m+e_m+e_k)
=\Delta(f_m)
=\varphi_1(f_m),
\]
where $k=q'-p'+m+1$ and $\varphi_1,\varphi_2$ are $\frac{1}{2}$-derivations associated with $\Delta$.

Then
\begin{equation}\label{efm}
  \sum_{j=s}^{t} \alpha_j f_{m+j} + \sum_{j=s}^{t} \alpha_j e_{m+j} + \sum_{j=p}^{q} \beta_j f_{m+j} + \sum_{j=s}^{t} \alpha_j e_{k+j} + \sum_{j=p}^{q} \beta_j f_{k+j} = \sum_{j=p'}^{q'} \beta'_j f_{m+j}.
\end{equation}

Assume that $\alpha_s \neq 0$ and $\alpha_t \neq 0$, and consider the individual basis elements $e_j, j \in \mathbb{Z}$ in the final equation above
\[
 \sum_{j=s}^{t} \alpha_j e_{m+j} + \sum_{j=s}^{t} \alpha_j e_{k+j}  = 0.
\]
The second sum in this equation represents a shift of the first sum by $q' - p' + 1 > 0$. From this, we deduce that $\alpha_s = 0$ and $\alpha_t = 0$, which is a contradiction. Consequently, it must hold that $\alpha_j = 0$ for all $s \le j \le t$. Substituting this result into equation \eqref{efm}, we obtain the following
\[
  \sum_{j=p}^{q} \beta_j f_{m+j} +  \sum_{j=p}^{q} \beta_j f_{k+j} = \sum_{j=p'}^{q'} \beta'_j f_{m+j}
\]
For this equality to hold, the lower bound of the first sum and the upper bound of the second sum on the left-hand side must coincide respectively with the lower and upper bounds of the sum on the right-hand side. This implies that $p = p'$ and $q = p' - 1$, which is a contradiction. Therefore, we conclude that $\beta'_j = 0$ for all $p' \le j \le q'$.
\end{proof}

\textit{Proof of Theorem \ref{locwitt(a,b)}}. Let $\Delta$ be any local $\frac12$-derivation on $\mathcal{W}(a,-1)$. Then there exists a $\frac12$-derivation $\varphi$ on $\mathcal{W}(a,-1)$ such that
$$\Delta(e_0)=\varphi(e_0)$$

Set $\Delta_1=\Delta-\varphi$. Then $\Delta_1$ is a local $\frac12$-derivation such that $\Delta_1(e_0)=0$.
By Lemma \ref{300}, $\Delta_1(e_m)=0$ for all $m\in\mathbb{Z}$ and by Lemma \ref{Lem400}, $\Delta_1(f_m)=0$ for all $m\in\mathbb{Z}$. Then $\Delta_1\equiv0.$
Thus $\Delta=\varphi$ is a $\frac12$-derivation.

Let $\Delta$ be any local $\frac12$-derivation on $\mathcal{W}(a,b)$ for $b \neq -1$. According to Theorem \ref{thm52}, the local $\frac{1}{2}$-derivation is trivial. Therefore, $\Delta$ is a rivial
$\frac{1}{2}$-derivation.  \hfill $\square$

\subsection{Local $\frac12$-derivations of thin Lie algebras}

Let us consider the following (see \cite{Kha}) so-called {\it thin Lie algebra} $\mathcal{L}$
with a basis \(\{e_n: n\in \mathbb{N}\}\), which is defined by the
following table of multiplications of the basis elements:
 $$[e_1,e_n]=e_{n+1},\ \ \ n\geq 2.$$
 and other products of the basis elements being zero.

\begin{lemma}\label{thin}
    Let $\varphi$ be a $\frac12$-derivation on $\mathcal L$.  Then
    \begin{equation*}
        \begin{split}
            \varphi(e_1)&=\sum\limits_{i=1}^n\alpha_ie_i,\quad \varphi(e_2)=\sum\limits_{i=2}^{m}\beta_ie_i,\\
   \varphi(e_j)&=((1-2^{2-j})\alpha_1+2^{2-j}\beta_2)e_j+2^{2-j}\sum\limits_{i=3}^{m}\beta_{i}e_{i+j-2},\ \ \ j\geq3.
        \end{split}
    \end{equation*}
where $n,m\in\mathbb{N}$ and $\alpha=(\alpha_1,...,\alpha_n)\in\mathbb{F}^n, \ \beta=(\beta_1,....,\beta_m)\in\mathbb{F}^m.$
   \end{lemma}

\begin{proof}  Let $\varphi$ be a $\frac12$-derivation of $\mathcal{L}.$  We set
\[\varphi(e_1)=\sum\limits_{i=1}^n\alpha_ie_i,\ \ \varphi(e_2)=\sum\limits_{i=1}^m\beta_ie_i,\ n,m\in\mathbb{N}.\]
We have
\begin{equation*}\begin{split}
\varphi(e_3)&=\varphi([e_1,e_2])=\frac{1}{2}([\varphi(e_1),e_2]+[e_1,\varphi(e_2)])\\
&=\frac{1}{2}
\left(\left[\sum\limits_{i=1}^n\alpha_ie_i,e_2\right]+\left[e_1,\sum\limits_{i=1}^{m}\beta_ie_i\right]\right)
=\frac12(\alpha_1+\beta_2)e_3+\frac12\sum\limits_{i=3}^{m}\beta_{i}e_{i+1}.
\end{split}\end{equation*}
Using $[e_2,e_3]=0,$ we have
\begin{equation*}\begin{split}
0=\varphi([e_2,e_3])&=\frac{1}{2}([\varphi(e_2),e_3]+[e_2, \varphi(e_3)])=\frac12\left(\left[\sum\limits_{i=1}^m\beta_ie_i,e_3 \right]\right)\\
&+\frac12\left(\left[e_2,\frac12(\alpha_1+\beta_2)e_3+\frac12\sum\limits_{i=1}^m\beta_{i+2}e_{i+3}\right]\right)=\frac12\beta_1e_4,
\end{split}\end{equation*}
Thus $\beta_1=0.$

Let  $$\varphi(e_j)=((1-2^{2-j})\alpha_1+2^{2-j}\beta_2)e_j+2^{2-j}\sum\limits_{i=3}^{m}\beta_{i}e_{i+j-2}, \ j\geq3,$$
 and
\begin{equation*}\begin{split}
\varphi(e_{j+1})&=\varphi([e_1,e_j])=\frac12([\varphi(e_1),e_j]+[e_1,\varphi(e_j)])\\
&=\frac12\left(\left[\sum\limits_{i=1}^n\alpha_ie_i, e_j\right]+\left[e_1,((1-2^{2-j})\alpha_1+2^{2-j}\beta_2)e_j+2^{2-j}\sum\limits_{i=3}^{m}\beta_{i}e_{i+j-2}\right]\right)\\
&=((1-2^{2-(j+1)})\alpha_1+2^{2-(j+1)}\beta_2)e_{j+1}+2^{2-(j+1)}\sum\limits_{i=3}^{m}\beta_{i}e_{i+(j+1)-2}.
\end{split}\end{equation*}

This means that the induction has been carried out.
\end{proof}
Let
$\alpha=(\alpha_1,\alpha_2,\dots)$ and $\beta=(\beta_1,\beta_2,\dots)$   be an arbitrary sequence from the field $\mathbb{F}$.
Define the operator
\[
D_{\alpha,\beta}:\mathcal L  \to \mathcal L
\]
which is a $\frac12$-derivation of $\mathcal L$, given on the basis elements by
\begin{equation*}
        \begin{split}
            D_{\alpha,\beta}(e_1)&=\sum\limits_{i=1}^n\alpha_ie_i,\quad D_{\alpha,\beta}(e_2)=\sum\limits_{i=2}^{m}\beta_ie_i,\quad n,m \in \mathbb{N},\\
   D_{\alpha,\beta}(e_j)&=((1-2^{2-j})\alpha_1+2^{2-j}\beta_2)e_j+2^{2-j}\sum\limits_{i=3}^{m}\beta_{i}e_{i+j-2},\ \ \ j\geq3.
        \end{split}
    \end{equation*}

We also define the operator $\Delta:\mathcal L\to\mathcal L$ by
\[
\Delta(e_1)=\Delta(e_2)=0,
\qquad
\Delta(e_j)=(1-2^{2-j})e_j,\quad j\ge3.
\]

\begin{theorem}
Let  $\mathcal L$ be the thin Lie algebra. Then $\mathcal{L}$ admits a local $\frac12$-derivation which is not a
$\frac12$-derivation.
\end{theorem}

\begin{proof}
Now let us define a linear operator $\Delta$ on $\mathcal L$ by
\begin{equation*}\label{locthin}
\Delta(e_1)=\Delta(e_2)=0,\quad \Delta(e_j)=(1-2^{2-j})e_j,\ j\geq 3.
\end{equation*}

We shall show that $\Delta$ is a local $\frac12$-derivation of $\mathcal L$, which is not a $\frac12$-derivation.

Firstly, we show that $\Delta$ is not a  $\frac12$-derivation. Let’s assume the opposite, $\Delta$ is a $\frac12$-derivation.

 Let us take $e_1$ and $e_3$ in $\mathcal L$.
Consider
$$\frac12\left([{\Delta}(e_1),e_3]+[e_1,{\Delta}(e_3)]\right)=\frac12\left[e_1,\frac12e_3\right]=\frac{1}{4}e_4.$$

On the other hand,

$${\Delta}([e_1,e_3])={\Delta}(e_4)=\frac{3}{4}e_4.$$

Comparing coefficients at the basis elements we obtain
$${\Delta}([e_1,e_3])\neq\frac{1}2\left([{\Delta}(e_1),e_3]+[e_1,{\Delta}(e_3)]\right)$$

So ${\Delta}$ is not a  $\frac12$-derivation.

Now we show that the $\Delta$ is a local $\frac12$-derivation.

 Let $x=\sum\limits_{i\in I}x_i e_i$ be an arbitrary element of $\mathcal L$,
where $x_i\neq0$ only for $i\in I\subset\mathbb{N}$. We shall find a $\frac12$-derivation $\varphi$, such that
 $\Delta(x)=\varphi(x)$.

\medskip
\noindent
\textbf{Case 1.} If $1\notin I$, put $\alpha=(1,0,0,\cdots)$ and $\beta=(0,0,\cdots).$ Then
\[
\Delta(x)= \sum_{i\in I}(1-2^{1-i})e_i = D_{\alpha,\beta}(x).
\]

\medskip
\noindent
\textbf{Case 2.} If $1\in I$, define
\[
\alpha(x)=\left(0,0,
\frac{(1-2^2)x_3}{x_1},\dots,\frac{\left(1-2^{1-k}\right)x_k}{x_1},\cdots\right),
\]
and
$$\beta=(0,0,\cdots).$$
Then
\[
\Delta(x)= \sum_{i\in I\setminus\{1,2\}}(1-2^{2-i})x_ie_i = D_{\alpha,\beta}(x).
\]
\end{proof}

\section{2-local $\frac12$-derivations of infinite-dimensional Lie algebras}

In this section, we study 2-local $\frac12$-derivations of infinite-dimensional Lie algebras.

\subsection{2-Local $\frac12$-derivations of Witt algebras}

Now we shall give the main result concerning 2-local $\frac12$-derivations of Witt algebras.

\begin{theorem}\label{Witt}
Any $2$-local $\frac{1}{2}$-derivation of Witt algebra $\mathcal{W}$ is a $\frac{1}{2}$-derivation.
\end{theorem}
\begin{proof} Let $\nabla$ be a $2$-local $\frac{1}{2}$-derivation on $\mathcal{W},$ such that $\nabla(e_0)=0.$
Then for any element $x=\sum\limits_{j\in\mathbb{Z}}x_je_j\in \mathcal{W},$ there exists a $\frac{1}{2}$-derivation $\varphi_{e_0,x}(x)$, such that
$$\nabla(e_0)=\varphi_{e_0,x}(e_0),\quad \nabla(x)=\varphi_{e_0,x}(x).$$

Hence,
$$0=\nabla(e_0)=\varphi_{e_0,x}(e_0)=\sum\limits_{j\in\mathbb{Z}}\alpha_je_{j},$$
which implies,  $\alpha_j=0,\ j\in\mathbb{Z}.$

Consequently, from the description of the $\frac{1}{2}$-derivation $\mathcal{W},$ we conclude that
$\varphi_{e_0,x}=0.$
Thus, we obtain that if $\nabla(e_0)=0,$ then
$
\nabla\equiv0.
$

Let now $\nabla$ be an arbitrary $2$-local $\frac{1}{2}$-derivation of \(\mathcal{W}\).
Take a $\frac{1}{2}$-derivation $\varphi_{e_0,x},$ such that
\begin{equation*}
\nabla(e_0)=\varphi_{e_0,x}(e_0)\ \ \text{and} \ \ \nabla(x)=\varphi_{e_0,x}(x).
\end{equation*}

Set $\nabla_1=\nabla-\varphi_{e_0,x}.$ Then $\nabla_1$ is a $2$-local
$\frac{1}{2}$-derivation, such that $\nabla_1(e_0)=0.$ Hence $\nabla_1(x)=0$ for all $x\in \mathcal{W},$  which implies $\nabla=\varphi_{e_0,x}.$
Therefore, $\nabla$ is a
$\frac{1}{2}$-derivation.
\end{proof}

\subsection{2-Local $\frac12$-derivations on subalgebras of the Witt algebra}

In this subsection, we study 2-local $\frac12$-derivations on positive Witt and one-sided Witt algebras.

\begin{theorem}\label{PWittloc}
Any 2-local $\frac12$-derivation of $\mathcal{W}^+$ and $\mathcal{W}_1$ is a $\frac12$-derivation.
\end{theorem}

\begin{proof} We prove the theorem for the algebra $\mathcal{W}^+$, since for the algebra $\mathcal{W}_1$ the proof is similar.

Let $\nabla$ be a $2$-local $\frac{1}{2}$-derivation on $\mathcal{W}^+,$ such that $\nabla(e_1)=0.$
Then for any element $x=\sum\limits_{j\in\mathbb{N}}x_je_j\in \mathcal{W}^+,$ there exists a $\frac{1}{2}$-derivation $\varphi_{e_1,x}(x)$, such that
$$\nabla(e_1)=\varphi_{e_1,x}(e_1),\quad \nabla(x)=\varphi_{e_1,x}(x).$$

Hence,
$$0=\nabla(e_1)=\varphi_{e_1,x}(e_1)=\sum\limits_{j\geq 1}\alpha_je_{j},$$
which implies,  $\alpha_j=0,\ j\in\mathbb{N}.$

Consequently, from the description of the $\frac{1}{2}$-derivation $\mathcal{W}^+,$ we conclude that
$\varphi_{e_1,x}=0.$
Thus, we obtain that if $\nabla(e_1)=0,$ then
$
\nabla\equiv0.
$

Let now $\nabla$ be an arbitrary $2$-local $\frac{1}{2}$-derivation of \(\mathcal{W}^+\).
Take a $\frac{1}{2}$-derivation $\varphi_{e_1,x},$ such that
\begin{equation*}
\nabla(e_1)=\varphi_{e_1,x}(e_1)\ \ \text{and} \ \ \nabla(x)=\varphi_{e_1,x}(x).
\end{equation*}

Set $\nabla_1=\nabla-\varphi_{e_1,x}.$ Then $\nabla_1$ is a $2$-local
$\frac{1}{2}$-derivation, such that $\nabla_1(e_1)=0.$ Hence $\nabla_1(x)=0$ for all $x\in \mathcal{W}^+,$  which implies $\nabla=\varphi_{e_1,x}.$
Therefore, $\nabla$ is a
$\frac{1}{2}$-derivation.
\end{proof}

\subsection{2-local $\frac{1}{2}$-derivation on Lie algebra $\mathcal{W}(a,b)$}

Now we shall give the main result concerning $2$-local $\frac{1}{2}$-derivation of Lie algebra $\mathcal{W}(a,b)$.

\begin{theorem}
Any $2$-local $\frac{1}{2}$-derivation of the Lie algebra $\mathcal{W}(a,b)$ is a $\frac{1}{2}$-derivation.
\end{theorem}
\begin{proof} Let $\nabla$ be a $2$-local $\frac{1}{2}$-derivation on $\mathcal{W}(a,-1),$ such that $\nabla(e_0)=0.$
Then for any element $x=\sum\limits_{n\in\mathbb{Z}}x_ne_n+\sum\limits_{m\in\mathbb{Z}}y_mf_m\in \mathcal{W}(a,-1)$ there exists a $\frac{1}{2}$-derivation $\varphi_{e_0,x}(x)$, such that
$$\nabla(e_0)=\varphi_{e_0,x}(e_0),\quad \nabla(x)=\varphi_{e_0,x}(x).$$

Hence,
$$0=\nabla(e_0)=\varphi_{e_0,x}(e_0)=\sum\limits_{t\in\mathbb{Z}}\alpha_te_{t}+\sum\limits_{t\in\mathbb{Z}}\beta_tf_{t}
,$$
which implies,  $\alpha_t=\beta_t=0,\ t\in\mathbb{Z}.$

Consequently, from the description of the $\frac{1}{2}$-derivation $\mathcal W(a,-1),$ we conclude that
$\varphi_{e_0,x}=0.$
Thus, we obtain that if $\nabla(e_0)=0,$ then
$
\nabla\equiv0.
$

Let now $\nabla$ be an arbitrary $2$-local $\frac{1}{2}$-derivation of $\mathcal W(a,-1)$.
Take a $\frac{1}{2}$-derivation $\varphi_{e_0,x},$ such that
\begin{equation*}
\nabla(e_0)=\varphi_{e_0,x}(e_0)\ \ \text{and} \ \ \nabla(x)=\varphi_{e_0,x}(x).
\end{equation*}

Set $\nabla_1=\nabla-\varphi_{e_0,x}.$ Then $\nabla_1$ is a $2$-local
$\frac{1}{2}$-derivation, such that $\nabla_1(e_0)=0.$ Hence $\nabla_1(x)=0$ for all $x\in \mathcal W(a,-1),$  which implies $\nabla=\varphi_{e_0,x}.$
Therefore, $\nabla$ is a
$\frac{1}{2}$-derivation.

Let $\nabla$ be any 2-local $\frac12$-derivation on $\mathcal{W}(a,b)$ for $b \neq -1$. According to Theorem \ref{thm52}, the 2-local $\frac{1}{2}$-derivation is trivial. Therefore, $\nabla$ is a trivial
$\frac{1}{2}$-derivation.
\end{proof}

\subsection{2-local $\frac12$-derivation on thin Lie algebras}

\begin{theorem}
Let  $\mathcal L$ be the thin Lie algebra. Then $\mathcal{L}$ admits a 2-local $\frac12$-derivation which is not a
$\frac12$-derivation.
\end{theorem}

\begin{proof}
For $x=\sum\limits_{i=1}^nx_ie_i\in \mathcal{L},\ n\in\mathbb{N}$ set
\begin{equation*}\label{yb4}
\nabla(x) = \begin{cases}
0, & \text{if $x_1=0$,}\\
\sum\limits_{i=2}^n2^{2-i}x_ie_i, & \text{if $x_1\neq0$}.
\end{cases}
\end{equation*}

We shall show that  $\nabla$ is a 2-local $\frac12$-derivation of
\(\mathcal{L}\), which is not a $\frac12$-derivation.

Firstly, we show that $\nabla$ is not a $\frac12$-derivation. Take the
elements  $x=e_1+e_2$ and $y=-e_1+e_2.$ We have
\[
\nabla(x+y)=\nabla(2e_2)=0\] and
\[
\nabla(x)+\nabla(y)=\nabla(e_1+e_2)+\nabla(-e_1+e_2)=e_2.\] Thus
\[
\nabla(x+y)\neq\nabla(x)+\nabla(y).\] So, \(\nabla\) is not
additive, and therefore it is not a $\frac12$-derivation.

Let us consider the linear maps $\varphi_1$ and $\varphi_2$ on \(\mathcal{L}\)
defined as:
\begin{equation}\label{yb5}
\varphi_1(e_i) = \begin{cases}
\sum\limits_{k=2}^n\alpha_ke_k, & \text{if $i=1$,}\\
0, & \text{if $i\geq2$},
\end{cases}
\end{equation}
where $\alpha_k\in \mathbb{F},$ $k=2, \ldots, n,$ and $n\in \mathbb{N}$ and
\begin{equation}\label{yb6}
\varphi_2(e_i) = \begin{cases}
0, & \text{if $i=1$,}\\
2^{2-i}e_i, & \text{if $i\geq2$}.
\end{cases}
\end{equation}
By Lemma \ref{thin}, it follows that both $\varphi_1$ and $\varphi_2$ are $\frac12$-derivations
of \(\mathcal{L}.\)

For any $x=\sum\limits_{k=1}^{n_x}  x_ke_k, \,
y=\sum\limits_{k=1}^{n_y}  y_ke_k \in \mathcal{L}$ we need to find
a $\frac12$-derivation $\varphi=\varphi_{x,y}$ such that
\begin{equation*}
\nabla(x)=\varphi(x)\ \ \text{and} \ \ \nabla(y)=\varphi(y).
\end{equation*}

It suffices to consider the following three cases.

\textbf{Case 1.} Let $x_1=y_1=0.$ In this case, we take $\varphi\equiv
0,$ because $\nabla(x)=\nabla(y)=0.$

\textbf{Case 2.} Let $x_1=0,\ y_1\neq0.$  In this case we take the
$\frac12$-derivation $\varphi_1$ of the form \eqref{yb5} with \(\displaystyle \alpha_1=0,
\alpha_k=\frac{2^{2-k}y_k}{y_1},\) \(2\leq k\leq n_y.\) Then
$$\nabla(x)=0=\varphi_1(x)$$ and
$$\nabla(y)=\sum\limits_{k=2}^{n_y} 2^{2-k}y_ke_k=y_1\sum\limits_{k=2}^{n_y}\frac{2^{2-k}y_k}{y_1}e_k=\varphi_1(y).$$
So, $\varphi$ is a $\frac12$-derivation such that $\nabla(x)=\varphi(x),\ \nabla(y)=\varphi(y).$

\textbf{Case 3.} Let $x_1\neq0,\ y_1\neq0.$ In this case we take the $\frac12$-derivation $\varphi_2$ of the form~\eqref{yb6}. Then
\[\nabla(x)=\sum\limits_{k=2}^{n_x}2^{2-k}x_ke_k=\varphi_2(x)\]
and
\[\nabla(y)=\sum\limits_{k=2}^{n_y}2^{2-k}y_ke_k=\varphi_2(y).\]

Therefore in all cases we constructed a derivation of  \(\mathcal{L}\) such that $\nabla(x)=\varphi(x),\ \nabla(y)=\varphi(y),$ i.e. $\nabla$ is a 2-local $\frac12$-derivation which is not a $\frac12$-derivation.
The proof is complete.
\end{proof}

\section{Local and 2-local $\frac{1}2$-derivations of the solvable Lie algebra with abelian nilpotent radical}

In this section, we study local and 2-local $\frac{1}2$-derivations of a solvable infinite-dimensional Lie algebra with an abelian radical of codimension 1.

Let us consider the following (see \cite{BKL}) so-called  the solvable Lie algebra with abelian nilpotent radical of codimension 1 $\mathfrak{g}$
with a basis \(\{e_i: i\in \mathbb{N}\}\), which is defined by the
following table of multiplications of the basis elements:
 $$[e_1,e_i]=e_i,\quad i\geq 2.$$
 and other products of the basis elements being zero.

   In  \cite{BKL}, the general structure of $\frac{1}2$-derivations of a solvable infinite-dimensional Lie algebra with an abelian radical of codimension 1 is described as follows.

\begin{lemma}\label{locsollem}
    Let $\varphi$ be a $\frac12$-derivation on $\mathfrak{g}$. Then
    $$\varphi(e_1)=\sum\limits_{i\in\mathbb{N}}\alpha_ie_i,\quad \varphi(e_i)=\alpha_1 e_i.$$
\end{lemma}

Now we shall give the main result concerning 2-local $\frac{1}{2}$-derivation of solvable Lie algebra with abelian nilpotent nilradical.

\begin{theorem}\label{solv}
Any $2$-local $\frac{1}{2}$-derivation of solvable Lie algebra with abelian nilpotent nilradical $\mathfrak{g}$ is a $\frac{1}{2}$-derivation.
\end{theorem}
\begin{proof} Let $\nabla$ be a $2$-local $\frac{1}{2}$-derivation on $\mathfrak{g},$ such that $\nabla(e_0)=0.$
Then for any element $x=\sum\limits_{j\in\mathbb{N}}x_je_j\in \mathfrak{g},$ there exists a $\frac{1}{2}$-derivation $\varphi_{e_1,x}(x)$, such that
$$\nabla(e_1)=\varphi_{e_1,x}(e_1),\quad \nabla(x)=\varphi_{e_1,x}(x).$$

Hence,
$$0=\nabla(e_1)=\varphi_{e_1,x}(e_1)=\sum\limits_{j\in\mathbb{N}}\alpha_je_{j},$$
which implies,  $\alpha_j=0,\ j\in\mathbb{N}.$

Consequently, from the description of the $\frac{1}{2}$-derivation $\mathfrak{g},$ we conclude that
$\varphi_{e_1,x}=0.$
Thus, we obtain that if $\nabla(e_1)=0,$ then
$
\nabla\equiv0.
$

Let now $\nabla$ be an arbitrary $2$-local $\frac{1}{2}$-derivation of \(\mathfrak{g}\).
Take a $\frac{1}{2}$-derivation $\varphi_{e_1,x},$ such that
\begin{equation*}
\nabla(e_1)=\varphi_{e_1,x}(e_1)\ \ \text{and} \ \ \nabla(x)=\varphi_{e_1,x}(x).
\end{equation*}

Set $\nabla_1=\nabla-\varphi_{e_1,x}.$ Then $\nabla_1$ is a $2$-local
$\frac{1}{2}$-derivation, such that $\nabla_1(e_1)=0.$ Hence $\nabla_1(x)=0$ for all $x\in \mathfrak{g},$  which implies $\nabla=\varphi_{e_1,x}.$
Therefore, $\nabla$ is a
$\frac{1}{2}$-derivation.
\end{proof}

Let
$\alpha=(\alpha_1,\alpha_2,\dots)$ be an arbitrary sequence from the field $\mathbb{F}$.
Define the operator
\[
D_{\alpha}:\mathfrak g \to \mathfrak g
\]
which is a $\frac12$-derivation of $\mathfrak g$, given on the basis elements by
\[
D_{\alpha}(e_1)=\sum_{i\in\mathbb N}\alpha_i e_i,
\qquad
D_{\alpha}(e_k)=\alpha_1 e_k,\quad k\ge2.
\]

We also define the operator $\overline{\Delta}:\mathfrak g\to\mathfrak g$ by
\[
\overline{\Delta}(e_1)=0,
\qquad
\overline{\Delta}(e_k)=e_k,\quad k\ge2.
\]

\begin{theorem}\label{locsolv}
For arbitrary local $\frac{1}{2}$-derivation $\Delta$ of the solvable Lie algebra with abelian nilpotent nilradical $\mathfrak{g}$ there exist a sequence $\alpha$ and a scalar $a\in\mathbb F$ such that
\[
\Delta = D_{\alpha} + a\,\overline{\Delta}.
\]

\end{theorem}

\begin{proof}
Since $\Delta$ is a local $\frac12$-derivation, for every
$x\in\mathfrak g$ there exists a sequence $\alpha(x)=\left(\alpha_1(x), \alpha_2(x), \cdots\right)$ such that
\[
\Delta(x)=D_{\alpha(x)}(x).
\]

First, consider $\Delta(e_1)=D_{\alpha(e_1)}(e_1)$ and define the operator
\[
\Delta_1(x)=\Delta(x)-D_{\alpha(e_1)}(x).
\]
Clearly,
\[
\Delta_1(e_1)=0.
\]

Now let $i\ge3$.
By the locality condition applied to the element $e_i+e_2$, we obtain
\[
\Delta_1(e_i+e_2)=\lambda_{e_i+e_2}(e_i+e_2).
\]
On the other hand,
\[
\Delta_1(e_i+e_2)=\Delta_1(e_i)+\Delta_1(e_2)
=\lambda_i e_i+\lambda_2 e_2.
\]
This implies $\lambda_i=\lambda_2=\lambda$.
Hence,
\[
\Delta_1(e_k)=\lambda e_k,\qquad k\ge2.
\]
Therefore, $\Delta_1=\lambda\overline{\Delta}$ and $a=\lambda$, which yields
\[
\Delta = D_{\alpha} + a\,\overline{\Delta}.
\]

It remains to show that $\overline{\Delta}$ is a local $\frac12$-derivation.
Let $x=\sum\limits_{i\in I}x_i e_i$ be an arbitrary element of $\mathfrak g$,
where $x_i\neq0$ only for $i\in I\subset\mathbb{N}$.

\medskip
\noindent
\textbf{Case 1.} If $1\notin I$, i.e. $x_1=0$ then
\[
\overline{\Delta}(x)= x = D_{(1,0,0,\dots)}(x).
\]

\medskip
\noindent
\textbf{Case 2.} If $1\in I$, $x_1\neq0$ define
\[
\alpha(x)=\left(0,\frac{x_2}{x_1},
\frac{x_3}{x_1},\dots\right).
\]
Then
\[
\overline{\Delta}(x)=\sum_{k\ge2, k\in I}x_k e_k=D_{\alpha(x)}(x).
\]

Thus, $\overline{\Delta}$ is a local $\frac12$-derivation.

\end{proof}

\textbf{My manuscript has no associated data}

\textbf{Ethical approval} Not applicable

\textbf{Conflict of interest} The authors declare no competing interests.

\textbf{Author Contributions:}All authors wrote, read, and approved the published version of the
manuscript

\end{document}